\documentclass[12pt]{article}
\usepackage{amsmath,amsthm,amscd,amsfonts}

 \newcounter{num}
 \newtheorem{theorem}{Theorem}[section]

\theoremstyle{plain}
\newtheorem{corollary}[theorem]{Corollary}

\newtheorem{lemma}[theorem]{Lemma}

\theoremstyle{remark}

\theoremstyle{definition}

\begin{document}

\title{A Note on Alternating Projections \\ 
       in Hilbert Space}
\author{Eva Kopeck\'a and Simeon Reich}
\date{}
\maketitle

\baselineskip 22pt plus 2pt

\newcommand{\be}        {\begin{eqnarray}}
\newcommand{\ee}        {\end{eqnarray}}
\newcommand{\pl}{\partial}
\newcommand{\sbs}{\subset}
\newcommand{\vr}{\varphi}
\newcommand{\lm}{\lambda}
\newcommand{\eps}{\varepsilon}
\newcommand{\nb}{\nabla}
\newcommand{\wt}{\widetilde}

\newcommand{\cA}      {{\mathcal A}}
\newcommand{\cM}      {{\mathcal M}}

\newcommand{\BB}       { \mathbb{B}}
\newcommand{\RR}       { \mathbb{R}}
\newcommand{\CC}       { \mathbb{C}}
\newcommand{\NN}       { \mathbb{N}}
\newcommand{\spn}{{\rm span\,}}
\newcommand{\dist}{{\rm dist\,}}

\begin{abstract}
We provide a direct proof of a result regarding the asymptotic behavior  
of alternating nearest point projections onto two closed and convex sets 
in Hilbert space. Our arguments are based on nonexpansive mapping theory.  
\end{abstract}

\medskip

\medskip


\bigskip
\bigskip

\section{Introduction}\label{Sec-intro}
\setcounter{equation}{0}

The purpose of this note is to provide a direct proof of a result 
regarding the asymptotic behavior of alternating nearest point projections 
onto two closed and convex subsets of a Hilbert space (see Theorem 
\ref{thm-dist} below). Our arguments are based on nonexpansive mapping 
theory. 

Let $S_1$ and $S_2$ be two closed subspaces of a real Hilbert space
$(H, \langle \cdot, \cdot \rangle)$ with induced norm $| \cdot |$, and let 
$P_1 : H \to S_1$ and
$P_2 : H \to S_2$ be the corresponding orthogonal projections of
$H$ onto $S_1$ and $S_2$, respectively. Denote by
$\NN = \{0, 1, 2, \dots \}$ the set of nonnegative integers.
Let $x_0$ be an arbitrary point in $H$, and define the sequence
$\{x_n : n \in \NN\}$ of alternating projections by
\begin{equation}
\label{equ-alter}
x_{2n+1} = P_1x_{2n} \text{\quad and \quad} x_{2n+2} = P_2x_{2n+1},
\end{equation}
where $n \in \NN$.

We begin by recalling von Neumann's classical theorem 
\cite[page 475]{Ne49}.

\begin{theorem}
\label{thm-vn}
The sequence $\{x_n : n \in \NN\}$ defined by (\ref{equ-alter})
converges in norm as $n \rightarrow \infty$ to $P_{S}x_0$, where $P_{S} :
H \to S$ is the
orthogonal projection of $H$ onto the intersection $S = S_1 \cap S_2$.
\end{theorem}

This theorem was rediscovered by several other authors; see, for example,
\cite{Ar50}, \cite{Na53} and \cite{Wi55}. More information regarding this
theorem and its diverse applications can be found in \cite{De01} and the
references mentioned therein. Several recent proofs of Theorem 
\ref{thm-vn} can be found, for instance, in \cite{BMR04}, \cite{KR04}, 
\cite{KR10} and \cite{KR11}.

Now let $C_1$ and $C_2$ be two nonempty, closed and convex subsets of  
$(H, \langle \cdot, \cdot \rangle)$, and let $P_1 : H \to C_1$ and 
$P_2 : H \to C_2$ be the corresponding nearest point projections of 
$H$ onto $C_1$ and $C_2$, respectively. 

Let $D$ be a subset of $H$. Recall that a mapping $T: D \to H$ is 
called {\em nonexpansive} (that is, $1$-Lipschitz) if 
$|Tx - Ty| \le |x - y|$ for all $x$ and $y$ in $D$.   

\begin{theorem}
\label{thm-breg}
Assume that the intersection $C = C_1 \cap C_2$ is not empty. 
Then the sequence $\{x_n : n \in \NN\}$ defined by (\ref{equ-alter})
converges weakly as $n \rightarrow \infty$ to $R_Cx_0$, where 
$R_C : H \to C$ is a nonexpansive retraction of $H$ onto $C$. 
\end{theorem}

This theorem is due to Bregman \cite[page 688]{Br65}. It is now known 
\cite{Hu04} that the sequence $\{x_n : n \in \NN\}$ of alternating nearest 
point projections does not, in general, converge in norm. In this 
connection, see also \cite{MR03} and \cite{Ko08}. 

When can we be sure that norm convergence does occur and what happens 
when the sets $C_1$ and $C_2$ are disjoint? In order to recall 
certain answers to these natural questions, 
we first denote by ${\rm d}(C_1, C_2)$ the distance between the sets $C_1$ 
and $C_2$, that is,
\begin{equation}
\label{equ-dist}
{\rm d}(C_1, C_2) := \inf\{|x - y| : x \in C_1, \; y \in C_2\}.
\end{equation}

Now we can quote \cite[Theorem 4.1]{KR04}. In this connection, see also 
\cite{BBR78}, \cite{BB93} and \cite{BB94}.  

\begin{theorem}
\label{thm-kr}
Let $C_1$ and $C_2$ be two nonempty, closed and convex subsets of 
the Hilbert space $H$, and let $P_1 : H \to C_1$   
and $P_2 : H \to C_2$ be the corresponding nearest point
projections of
$H$ onto $C_1$ and $C_2$, respectively.
Let the sequence $\{x_n : n \in \NN\}$ be defined by (\ref{equ-alter}). 
\begin{enumerate}
\item [(a)] If ${\rm d}(C_1,C_2)$ is attained, 
then the sequence $\{x_{2n} : n \in \NN\}$ converges 
weakly as $n \rightarrow \infty$ to a fixed point z of $P_2P_1$ and 
$\{x_{2n+1} : n \in \NN\}$ converges weakly 
as $n \rightarrow \infty$ to $P_1z$.
\item[(b)] If ${\rm d}(C_1,C_2)$ is not attained, then the 
sequence $\{|x_n| : n \in \NN\} \rightarrow \infty$ as $n \rightarrow \infty$.
\item[(c)] If both $C_1$ and $C_2$ are symmetric with respect to the 
origin, then the sequence $\{x_n : n \in \NN\}$ converges in norm as 
$n \rightarrow \infty$ to a point in the intersection $C = C_1 \cap C_2$.
\end{enumerate}
\end{theorem}

Part (c) of Theorem \ref{thm-kr} seems to be a good nonlinear analogue 
of von Neumann's linear Theorem \ref{thm-vn}. 

Now we are ready to state the result which is of concern to us in the 
present note. 

\begin{theorem}
\label{thm-dist}
Let $C_1$ and $C_2$ be two nonempty, closed and convex subsets of a real
Hilbert space $(H, \langle \cdot, \cdot \rangle)$ 
with induced norm $| \cdot |$, and let $P_1 : H \to C_1$ 
and $P_2 : H \to C_2$ be the corresponding nearest point 
projections of
$H$ onto $C_1$ and $C_2$, respectively.
Let the sequence $\{x_n : n \in \NN\}$ be defined by (\ref{equ-alter}).
Then 
\begin{equation}
\label{equ-thm}
\lim_{n \rightarrow \infty}|x_{2n+2} - x_{2n+1}| = \lim_{n \rightarrow 
\infty}|x_{2n+1} - x_{2n}| = {\rm d}(C_1, C_2). 
\end{equation}
\end{theorem}

This theorem was obtained in \cite[page 433]{BB94} as a consequence of the 
authors' analysis of Dykstra's algorithm. As we have already mentioned, 
our goal is to provide a
direct proof which is based on nonexpansive mapping theory. This proof is
given in Section \ref{Sec-Alternating}. In the next section we present    
several lemmata which are used in our proof.

We remark in passing that since, by the definition of the nearest point 
projection, we have  
\begin{equation}
|x_{2n+2} - x_{2n+1}| \le |x_{2n+1} - x_{2n}| \le |x_{2n} - x_{2n-1}| 
\le |x_{2n-1} - x_{2n-2}|
\end{equation}
for all $n \ge 2$, both sequences 
$\{|x_{2n+2} - x_{2n+1}|\}^{\infty}_{n=0}$
and 
$\{|x_{2n+1} - x_{2n}|\}^{\infty}_{n=1}$
{\em decrease} to their common limit. 

We conclude this introduction with a corollary of Theorem \ref{thm-dist}. 
It can be proved by appealing, for example, to the parallelogram law.  
In this connection, see also \cite[page 433]{BB94}. 

\begin{corollary}
\label{cor-least}
In the setting of Theorem \ref{thm-dist}, we have  
\begin{equation}
\label{equ-conv}
\lim_{n \rightarrow \infty}(x_{2n+2} - x_{2n+1}) 
= \lim_{n \rightarrow \infty}(x_{2n} - x_{2n+1}) = v, 
\end{equation}
where $v$ is the point of least norm in the closure of $C_2 - C_1$
and the convergence is in norm. 
\end{corollary}
\begin{proof}
Let $K$ be the closure of $C_2 - C_1$ and consider, for instance, the 
sequence $\{y_n : n \in \NN\} \subset K$ defined by 
$y_n : = x_{2n+2} - x_{2n+1}$, where $n \in \NN$. Applying the 
parallelogram law to $y_n$ and $v$, we obtain 
\begin{equation}
|y_n + v|^2 + |y_n - v|^2 = 2(|y_n|^2 + |v|^2) 
\end{equation}
and
\begin{equation}
|(y_n -v)/2|^2 = (|y_n|^2 + |v|^2)/2 - |(y_n + v)/2|^2. 
\end{equation}  
Since $(y_n + v)/2 \in K$, we know that $|(y_n + v)/2| \ge |v|$. 
Hence 
\begin{equation} 
|y_n - v|^2 \le 2(|y_n|^2 - |v|^2)
\end{equation}
for each $n \in \NN$. The result now follows from Theorem \ref{thm-dist}. 
\end{proof}

\section{Nearest Point Projections}\label{Sec-Nearest}
\setcounter{equation}{0}
Preparing for the proof of Theorem \ref{thm-dist}, we first present 
several known facts regarding nearest point projections onto 
closed and convex sets in Hilbert spaces. 

\begin{lemma}
\label{lem-proj}
Let $C$ be a closed and convex subset of the Hilbert space $H$, 
and let $P_C : H \to C$ be the nearest point projection of $H$ onto 
$C$. Then 
\begin{equation}  
\label{equ-proj}
|y - P_Cx|^2 + |x - P_Cx|^2 \le |x - y|^2
\end{equation}
for all $x \in H$ and $y \in C$.
\end{lemma}
\begin{proof}
We have 
\begin{equation} \nonumber
\begin{split}
|x - y|^2  
& = |x -P_Cx + P_Cx -y|^2 \\  
& = |x - P_Cx|^2 + 2\langle x - P_Cx, P_Cx - y \rangle + |y - P_Cx|^2 \\
& \ge |x - P_Cx|^2 + |y - P_Cx|^2,
\end{split}
\end{equation}
as asserted. 
\end{proof}

Let $D$ be a subset of $H$ and let $T : D \to H$ be a mapping. Recall 
that the mapping $T$ is said to be {\em asymptotically regular} at 
$x \in D$ (see, for example, \cite{BBR78}) if it can be iterated at $x$ 
and $\lim_{n \rightarrow \infty}
(T^nx - T^{n+1}x) = 0$. We now quote \cite[Theorem 3.1, page 144]{Ba02}.  
In this connection, see also \cite[Theorem 4.6, page 8]{BMMW12}.

\begin{lemma}
\label{lem-ar}
The compositon of finitely many nearest point projections onto closed 
and convex subsets of a Hilbert space is asymptotically regular. 
\end{lemma}

Our next lemma can be found in \cite[page 449]{CG59}. 

\begin{lemma}
\label{lem-d}
Let ${\rm d}(C_1, C_2)$ be the distance between two closed and convex 
subsets $C_1$ and $C_2$ of a Hilbert space $H$, and let 
$P_1 : H \to C_1$ and $P_2 : H \to C_2$ be the corresponding 
nearest point projections of $H$ onto $C_1$ and $C_2$, respectively.   
Then the fixed point set ${\rm F}(P_2P_1)$ of the composition $P_2P_1$ 
in $H$ coincides with the set 
$\{b \in C_2 : |b - P_1b| = {\rm d}(C_1, C_2)\}$. 
\end{lemma}

If $D$ is a subset of $H$, then a mapping $T: D \to H$ is said to be 
{\em strongly nonexpansive} \cite{BR77} if it is nonexpansive and 
$\{(x_n - y_n) - (Tx_n - Ty_n)\} \rightarrow 0$ as $n \rightarrow \infty$ 
whenever $\{x_n\}$ and 
$\{y_n\}$ are two sequences in $D$ such that the sequence $\{x_n - y_n\}$ 
is bounded and $\{|x_n - y_n| - |Tx_n - Ty_n|\} \rightarrow 0$ 
as $n \rightarrow \infty$.

\begin{lemma}
\label{lem-sne}
Every nearest point projection is strongly nonexpansive. 
\end{lemma}
\begin{proof}
This assertion is true because every nearest point projection in Hilbert 
space is known to be firmly nonexpansive (see \cite[page 18]{GR84}) and 
every firmly nonexpansive mapping in a uniformly convex Banach space is 
strongly nonexpansive by \cite[Proposition 2.1, page 463]{BR77}. 

\end{proof}

\section{Alternating Projections}\label{Sec-Alternating}
\setcounter{equation}{0}

\begin{proof} [Proof of Theorem \ref{thm-dist}]
Since the composition $P_2P_1$ is asymptotically regular (see Lemma 
\ref{lem-ar}), the sequence  
$\{x_{2n+2} - x_{2n}\} \rightarrow 0$ as $n \rightarrow \infty$ and 
therefore the evaluation of the second limit follows from the evaluation 
of the first one. Alternatively, we can simply interchange the roles of 
$C_1$ and $C_2$.

In order to evaluate the first limit, assume initially that the 
distance ${\rm d}(C_1, C_2)$ between the closed and convex sets $C_1$ 
and $C_2$ is attained. Then we know (see Lemma \ref{lem-d}) that 

\begin{equation}
\label{equ-fixedset}
{\rm F}(P_2P_1) = \{b \in C_2 : |b - P_1b| = {\rm d}(C_1, C_2)\}.
\end{equation}

Let $z \in F(P_2P_1)$. Then $|z - P_1z| = {\rm d}(C_1, C_2)$.
Since 
$$|x_{2n+2} - z| = |P_2x_{2n+1} - P_2P_1z| \le |x_{2n+1} - P_1z| 
= |P_1x_{2n} - P_1z| \le |x_{2n} - z|$$
and the sequence $\{|x_{2n} - z|\}_{n=1}^{\infty}$ is decreasing, 
we see that
\begin{equation}
\label{equ-sne}
\lim_{n \rightarrow \infty}(|x_{2n+1} - P_1z| - |P_2x_{2n+1} - P_2P_1z|)
= 0.
\end{equation}
Since the projection $P_2$ is strongly nonexpansive (see 
Lemma \ref{lem-sne}), it follows that 
\begin{equation}
\lim_{n \rightarrow \infty}(x_{2n+1} - x_{2n+2}) = P_1z - z. 
\end{equation}
Hence
\begin{equation}
\label{equ-claim}
\lim_{n \rightarrow \infty}|x_{2n+2} - x_{2n+1}| = |z - P_1z| 
= {\rm d}(C_1, C_2), 
\end{equation}
as claimed. 

Now assume that ${\rm d}(C_1, C_2)$ is not necessarily attained. In this 
case, there is, however, a sequence $\{z_n : n \in \NN\} \subset C_2$ such 
that the sequence $\{|z_n - P_1{z_n}|\} \rightarrow {\rm d}(C_1, C_2)$ 
as $n \rightarrow \infty$.  

Applying Lemma \ref{lem-proj} to $C_2$ and $P_2$, we obtain, for each 
$n \in \NN$,  
\begin{equation}
|z_n - P_2P_1z_n|^2 + |P_1z_n - P_2P_1z_n|^2 \le |P_1z_n - z_n|^2. 
\end{equation}
Hence $|z_n -P_2P_1z_n|^2 + {\rm d}(C_1, C_2)^2 \le |P_1z_n - z_n|^2$
for each $n \in \NN$. Consequently, the sequence 
$\{|z_n - P_2P_1z_n|\} \rightarrow 0$ as $n \rightarrow \infty$. 
We also have
\begin{equation} \nonumber
\begin{split}
|x_{2n+2} - z_n| 
& \le |P_2x_{2n+1} - P_2P_1z_n| + |P_2P_1z_n - z_n| \\
& \le |x_{2n+1} - P_1z_n| + |P_2P_1z_n - z_n| \\
& \le |x_{2n} - z_n| + |P_2P_1z_n - z_n|.
\end{split}
\end{equation}

Observing that  
\begin{equation}
\big| |x_{2n+2} - z_n| - |x_{2n} - z_n| \big| \le |x_{2n+2} - x_{2n}|
\end{equation}
and that $\lim_{n \rightarrow \infty} |x_{2n+2} - x_{2n}| = 0$ 
(because the composition $P_2P_1$ is asymptotically regular by Lemma 
\ref{lem-ar}), we obtain
\begin{equation}
\lim_{n \rightarrow \infty} 
\big[|x_{2n+1} - P_1z_n| - |P_2x_{2n+1} - P_2P_1z_n|\big] = 0.  
\end{equation}

At this point we invoke Lemma \ref{lem-sne} once more to obtain that 

\begin{equation}
\lim_{n \rightarrow \infty}\big[x_{2n+1} - P_1z_n - (x_{2n+2} - 
P_2P_1z_n)\big] = 0,
\end{equation}

\begin{equation}
\lim_{n \rightarrow \infty}\big[x_{2n+1} - P_1z_n - (x_{2n+2} - 
z_n)\big] = 0 
\end{equation}

and 
\begin{equation}
\lim_{n \rightarrow \infty}\big[(x_{2n+1} - x_{2n+2}) - (P_1z_n - 
z_n)\big] = 0.
\end{equation}

Hence 
\begin{equation}
\label{assert}
\lim_{n \rightarrow \infty} |x_{2n+2} - x_{2n+1}| = {\rm d}(C_1, C_2),
\end{equation}
as asserted.
\end{proof}

\text{}

\text{}

\noindent
{\bf Acknowledgements}. 

\noindent The first author was supported by Grants FWF P23628-N18, 
GA\v{C}R P201/12/0290 and by RVO 67985840. The second author was 
partially supported by the Israel Science Foundation (Grant 389/12), the 
Fund for the Promotion of Research at the Technion (Grant 2001893), and by 
the Technion General Research Fund (Grant 2016723).

\text{}

\text{}

\bigskip

\text{}

\text{}

\bigskip

\bigskip

\newpage

\noindent {\bf Addresses} \\

\noindent
Eva Kopeck\'a
\newline
Institute of Mathematics, Czech Academy of Sciences,
\v{Z}itn\'a 25, CZ-11567 Prague, Czech Republic
\newline
and
\newline
Department of Mathematics, University of Innsbruck,
Technikerstrasse 19a,
A-6020 Innsbruck, Austria

E-mail address: kopecka@math.cas.cz  \\

\noindent
Simeon Reich
\newline
Department of Mathematics, The Technion--Israel Institute of
Technology, 32000 Haifa, Israel

E-mail address: sreich@tx.technion.ac.il


\bigskip

\begin{thebibliography} {XXX}

\bibitem [Ar50] {Ar50} N.~Aronszajn, {\it Theory of reproducing kernels},
Trans. Amer. Math. Soc. {\bf 68} (1950), 337--403.

\bibitem [BBR78] {BBR78} J.-B.~Baillon, R.~E. Bruck, S.~Reich,
{\it On the asymptotic behavior of nonexpansive mappings and semigroups 
in Banach spaces}, Houston J. Math. {\bf 4} (1978), 1--9.  

\bibitem [Ba02] {Ba02} H.~H.~Bauschke, {\it The composition of projections 
onto closed convex sets in Hilbert space is asymptotically regular}, 
Proc. Amer. Math. Soc. {\bf 131} (2002), 141--146.  

\bibitem [BB93] {BB93} H.~H.~Bauschke, J.~M.~Borwein,
{\it On the convergence of von Neumann's alternating projection algorithm 
for two sets}, Set-Valued Anal. {\bf 1} (1993), 185--212.

\bibitem [BB94] {BB94} H.~H.~Bauschke, J.~M.~Borwein,
{\it Dykstra's alternating projection algorithm for two sets},
J. Approximation Theory {\bf 79} (1994), 418--443.  

\bibitem [BMMW12] {BMMW12} H.~H. Bauschke, V. Mart\'{\i}n M\'arquez,  
S. M. Moffat, X. Wang, {\it Compositions and convex combinations of 
asymptotically regular firmly nonexpansive mappings are also 
asymptotically regular}, Fixed Point Theory Appl. {\bf 2012} (2012), 
1--11.

\bibitem [BMR04] {BMR04} H. H.~Bauschke, E.~Matou\v{s}kov\'{a}, S.~Reich, 
{\it Projection and proximal point methods: convergence results and 
counterexamples}, Nonlinear Anal. {\bf 56} (2004), 715--738. 

\bibitem [Br65] {Br65} L. M. Bregman, {\it Finding the common point of 
convex sets by the method of successive projection},  
Soviet Math. Dokl {\bf 6} (1965), 688--692. 

\bibitem [BR77] {BR77} R. E. Bruck, S. Reich, 
{\it Nonexpansive projections and resolvents of accretive operators in 
Banach spaces}, Houston J. Math. {\bf 3} (1977), 459--470. 

\bibitem [CG59] {CG59} W. Cheney, A. A. Goldstein, {\it Proximity maps for 
convex sets}, Proc. Amer. Math. Soc. {\bf 10} (1959), 448--450.

\bibitem [De01] {De01} F. Deutsch, {\it Best Approximation in Inner 
Product Spaces}, Springer, New York, NY, 2001.

\bibitem [GR84] {GR84} K. Goebel, S. Reich, {\it Uniform Convexity, 
Hyperbolic Geometry, and Nonexpansive Mappings}, Marcel Dekker, New York 
and Basel, 1984. 

\bibitem [Hu04] {Hu04} H. S. Hundal, {\em An alternating projection that 
does not converge in norm}, Nonlinear Anal. {\bf 57} (2004), 35--61.  

\bibitem [Ko08] {Ko08} E. Kopeck\'a, {\it Spokes, mirrors and alternating 
projections}, Nonlinear Anal. {\bf 68} (2008), 1759--1764. 

\bibitem [KR04] {KR04} E.~Kopeck\'a, S.~Reich, {\it A note on the 
von Neumann alternating projections algorithm}, 
J. Nonlinear Convex Anal. {\bf 5} (2004), 379--386.

\bibitem [KR10] {KR10} E.~Kopeck\'a, S.~Reich, {\it Another note 
on the von Neumann alternating projections algorithm},
J. Nonlinear Convex Anal. {\bf 11} (2010), 455--460. 

\bibitem [KR11] {KR11} E.~Kopeck\'a, S.~Reich, {\it 
Alternating projections and orthogonal 
decompositions}, J. Nonlinear Convex Anal. {\bf 12} 
(2011), 155--159. 

\bibitem [MR03] {MR03} E. Matou\v{s}kov\'a, S. Reich, {\it The Hundal 
example revisited}, J. Nonlinear Convex Anal. {\bf 4} (2003), 411--427. 

\bibitem [Na53] {Na53} H.~Nakano, {\it Spectral Theory in Hilbert Space}, 
Japan Soc. Promotion Sci., Tokyo, 1953. 

\bibitem [Ne49] {Ne49} J.~von Neumann, {\it On rings of operators. 
Reduction theory}, Ann. of Math. {\bf 50} (1949), 401--485. 

\bibitem [Wi55] {Wi55} N. Wiener, {\it On the factorization of matrices},   
Commentarii Math. Helv. {\bf 29} (1955), 97--111. 

\end{thebibliography}
\end{document}